\documentclass{amsart}

\usepackage{graphicx} %
\usepackage{amsfonts, amsmath, amsthm, amssymb}
\usepackage{mathtools}
\usepackage[dvipsnames]{xcolor}
\usepackage[
    pdftex=true,
    colorlinks=true,
    linkcolor=RoyalPurple,   %
    citecolor=Blue,     %
    urlcolor=Violet     %
]{hyperref}
\usepackage[margin=1.25in]{geometry}
\usepackage{float, comment}
\usepackage{algorithm}
\usepackage{algorithmic}
\usepackage{cleveref}
\usepackage[utf8]{inputenc}
\usepackage[T1]{fontenc}
\usepackage{empheq}
\usepackage{amsfonts}
\usepackage{tikz-cd}
\usepackage{caption}
\usepackage{subcaption}
\usepackage{enumitem}
\usepackage{url}

\def\det{{\mathrm{det}}}
\def\T{{\top}}
\def\tr{{\operatorname{Trace}}}
\def\phi{{\varphi}}

\def\R{{\mathbb R}}

\def\N{{\mathbb N}}

\def\cal{\mathcal}

\def\cH{{\cal H}}

\def\det{\operatorname{det}}

\def\dim{\operatorname{dim}}

\def\O{\Omega}

\usepackage[backend=bibtex, style=alphabetic]{biblatex}
\theoremstyle{plain}
\newtheorem{theorem}{Theorem}[section]
\newtheorem{lemma}[theorem]{Lemma}

\theoremstyle{definition}
\newtheorem{definition}[theorem]{Definition}

\def\p{{\hat{p}}}
\newcommand{\MSE}[1]{{{\rm MSE}\left[{#1}\right]}}
\newcommand{\Exp}[1]{{\mathbb{E}\left[{#1}\right]}}
\newcommand{\Bias}[1]{{{\rm Bias}\left[{#1}\right]}}
\newcommand{\Var}[1]{{{\rm Var}\left[{#1}\right]}}

\title{Density Estimation on Rectifiable Sets}
\author{Jack Kendrick}
\address{Department of Mathematics, University of Washington}
\email{jackgk@uw.edu}
\thanks{The author would like to thank Dmitriy Drusvyatskiy for helpful discussions regarding the content of this paper.}
\date{\today}

\begin{document}

\begin{abstract}
    Kernel density estimation is a popular method for estimating unseen probability distributions. However, the convergence of these classical estimators to the true density slows down in high dimensions. Moreover, they do not define meaningful probability distributions when the intrinsic dimension of data is much smaller than its ambient dimension. We build on previous work on density estimation on manifolds to show that a modified kernel density estimator converges to the true density on $d-$rectifiable sets. As a special case, we consider algebraic varieties and semi-algebraic sets and prove a convergence rate in this setting. We conclude the paper with a numerical experiment illustrating the convergence of this estimator on sparse data.
\end{abstract}

\maketitle

\section{Introduction}

Density estimation is a canonical problem in statistics: given $n$ random vectors $X_1, \ldots, X_n\in \R^D$ sampled independently from a probability distribution $P,$ we would like to estimate the density of $P$ at some unseen point $x\in\R^D.$ One popular method for estimating unseen density functions is through the use of a {\em kernel density estimator}:
\begin{equation}\label{eq:parzen-kde}
    {\mathrm{KDE}}_n(x) = \frac{1}{h_n^D}\sum_{i=1}^nK\left(\frac{\|X_i-x\|}{h_n}\right)
\end{equation}
where $K$ is some kernel function and $h_n$ is a user-defined bandwidth parameter. The kernel density estimator defines a new probability distribution on $\R^D$ that we hope converges to the true distribution $P$ as the number of samples $n$ increases. Indeed, foundational work by Parzen \cite{parzen} (when $D=1$) and Cacoullos \cite{cacoullos1966estimation} (when $D>1$) showed that when the bandwidth parameter $h_n$ is chosen appropriately so that $\lim_{n\to\infty}h_n= 0$ and $\lim_{n\to\infty}nh_n^d=\infty,$ then the estimator $\mathrm{KDE}_n$ converges to the true density $p.$ In particular, when $h_n$ is chosen in proportion to $n^{-1/(D+4)},$ the mean square error (MSE) of $\p$ is of the order
\begin{equation}
    \MSE{\mathrm{KDE}_n(x)} = O\left(\frac{1}{n^{4/(D+4)}}\right)
\end{equation}
as $n\to\infty.$ So, while the estimator $\mathrm{KDE}_n(x)$ does well-approximate the true density $p(x)$ as the number of samples increases, this convergence slows down as the ambient dimension of the data $D$ increases.

In many modern applications, although data is sampled in some high dimensional space $\R^D,$ we see that the {\em intrinsic} dimension of the data is much smaller. In these cases, the support of the probabiliy measure $P$ is some domain $\O\subset\R^D$ with dimension $d$ that is much smaller than the ambient dimension $D.$ Now, the kernel density estimator defined in \eqref{eq:parzen-kde} no longer defines a meaningful probability distribution on $\O:$ in order to account for the low-dimensionality of $\O,$ we must normalize the estimator to be zero outisde of $\O$ and infinite on $\O.$ One proposed workaround is to define the estimator

\begin{equation}\label{eq:p-hat}
    \p_n(x) = \frac{1}{h_n^d}\sum_{i=1}^nK\left(\frac{\|X_i-x\|}{h_n}\right),
\end{equation}

a modified version of the estimator $KDE_n$ that uses the intrinsic dimension $d$ of the data, rather than its ambient dimension $D.$ We expect this new estimator to approximate the density of each point $x\in\O$ as the number of samples $n$ grows.

The setting where $\O$ is a $d-$dimensional Riemannian submanifold of $\R^D$ was studied in \cite{pelletier2005riemannian} (when $\O$ is known), \cite{ozakin2009submanifold} (when only $\dim \O = d$ is known) and \cite{berenfeld2020submanifold} (where $\dim \O$ is possibly unknown). It was shown that when the bandwidth parameters $h_n$ are chosen proportionally to $n^{-1/(d+4)},$ then the mean square error $\MSE{\p_n(x)}$ is of the order $O\left(\frac{1}{n^{4/(d+4)}}\right)$ as $n\to\infty,$ mirroring the convergence rate of $\mathrm{KDE}_n.$

However, the assumption that data lies on a manifold is very strong and is unreasonable in many important examples. For example, a natural way to ensure vector data has a low intrinsic dimension is to assume some level of sparsity. In a similar vein, if data is matrix valued, we may assume that the data is low rank or positive semi-definite. The sets of sparse vectors, low rank matrices, and positive semi-definite matrices do not form yet do have intrinsic senses of dimension that should be exploited when estimating densities on these spaces. In this paper, we study the estimator $\p_n(x)$ on {\em $d-$rectifiable sets}, which can be thought of as the natural measure-theoretic generalizations of $d-$dimensional submanifolds and include the aforementioned examples as special cases. An important property of rectifiable sets is that they have {\em approximate tangent spaces} at almost every point. When these approximate tangent spaces satisfy a certain condition, discussed in Section \ref{sec:proof-rectifiable}, we prove the following theorem:

\begin{theorem}[KDE on rectifiable sets]\label{thm:main-rectifiable}
    Let $\O\subset\R^D$ be a $d-$rectifiable set %
    and $P$ a probability measure on $\O$ with density $p$ with respect to the $d-$dimensional Hausdorff measure $\cH^d.$ Assume that the kernel $K$ is continuous and vanishes outside of $[0, 1]$ Then, when bandwidth parameters $h_n$ are chosen appropriately, %
    the equality
    \begin{equation}\label{eq:convergence-rate}
        \MSE{\p_n(x)} = O\left(\frac{1}{n^{2m/(d+2m)}}\right).
    \end{equation}
    holds with probability 1, where $m$ is a measure of how well $\O$ is approximated by its approximate tangent spaces.
\end{theorem}

Note that this theorem mirrors the results of \cite{berenfeld2020submanifold,ozakin2009submanifold,pelletier2005riemannian} in a more general setting. As a consequence of this theorem, we prove that the convergence rate achieved in these previous works holds when $\O$ is locally a smooth manifold.\footnote{While our results are stated for sets that are locally smooth manifolds since we are interested in examples such as semi-algebraic sets, the hypothesis can be weakened to spaces that are locally only $C^3$ without changing the proof. Minor modifications may be made to the proof so that the result holds on spaces that are locally only $C^2$.} We emphasise that although this result can also be achieved with minor modifications to the arguments in \cite{berenfeld2020submanifold,ozakin2009submanifold,pelletier2005riemannian}, our perspective that this is a special case of a more general setting is novel.

\begin{theorem}[KDE on a.e. smooth spaces]\label{thm:settings}
    Suppse $\O$ is locally a smooth manifold almost everywhere. Assume that $K$ is twice differentiable and vanishes outside $[0, 1]$ and the density $p$ is twice differentiable. Then, when the bandwidth parameters $h_n$ are chosen in proportion to $n^{-1/(d+4)},$ the equality
    \begin{equation}\label{eq:convergence-rate-smooth}
       \MSE{\p_n(x)} = O\left(\frac{1}{n^{4/(d+4)}}\right).
    \end{equation}
    holds with probability 1.
\end{theorem}
The proof of Theorem \ref{thm:settings} amounts to showing that a neighbourhood of $x\in\O$ is approximated by the tangent space $T_x\O$ with only $O(h_n^2)$ error. Note that since algebraic varieties and semi-algebraic sets are smooth almost everywhere, this includes the case that $\O$ is a $d-$dimensional algebraic variety or semi-algebraic set.

\subsection*{Outline} The rest of this section contains a brief overview of related literature. In Section \ref{sec:prelim}, we discuss all necessary background material on $d-$rectifiable sets. The proofs of Theorems \ref{thm:main-rectifiable} and \ref{thm:settings} are contained in Section \ref{sec:proof-rectifiable} and Section \ref{sec:settings} respectively. We conclude with Section \ref{sec:numerics} which details a numerical experiment using $\p_n$ to estimate a known density function on sparse vectors. 

\subsection*{Related Literature} Although little attention has been paid to density estimation on non-smooth spaces, there is a wealth of literature on nonparametric density estimation on manifolds, beginning with \cite{hendriks}. The work \cite{hendriks} uses the techniques of Fourier analysis to estimate probability densities on manifolds, yet is limited since it requires knowledge of the eigenfunctions of the Laplace-Beltrami operator on the space. The kernel density estimator \eqref{eq:p-hat} on Riemannian submanifolds was first studied in \cite{pelletier2005riemannian}, but again is limited since knowledge of the space is required. This was generalized to unknown submanifolds by \cite{ozakin2009submanifold}. The analysis in \cite{ozakin2009submanifold} relies on the observation that compactly supported kernels essentially become delta functions as the bandwidth parameter approaches zero. This fact was first observed in \cite{diffusion} in the context of estimating the Laplace-Beltrami operator on manifolds. The setting where data lies on a manifold with boundary was studied in \cite{berry-sauer}, where an additional estimator known as the boundary direction estimator is used to renormalize kernels so that density estimations are consistent as points approach the boundary of the space without needing prior knowledge of the boundary's location. 

Our analysis in the proofs of Theorems \ref{thm:main-rectifiable} and \ref{thm:settings} relies on the existence of (approximate) tangent spaces to the domain $\O.$ The work \cite{kim-park} also focuses on the existence of tangent spaces and introduces a kernel density estimator that is defined on the tangent space to the manifold $\O.$ Although it is shown that this estimator does converge to the true denstiy on the domain, this approach is severely limited since it requires knowledge and evaluation of the exponential map on $\O.$ 

More recently, kernel density estimation on unknown submanifolds without boundary was studied in \cite{berenfeld2020submanifold}. Under certain conditions relating to the reach of the submanifold, \cite{berenfeld2020submanifold} provides convergence rates of the kernel density estimator that depend on the smoothness of the probability density function. This convergence rate mirrors the convergence rate in our Theorem \ref{thm:main-rectifiable}, however our result is dependent on the smoothness of the data's domain and not the probability density function itself. The analysis in \cite{berenfeld2020submanifold} uses local parametrizations of the space by its tangent space, much like our arguments in the proof of Theorem \ref{thm:settings}. We emphasise that our main result Theorem \ref{thm:main-rectifiable} is not predicated on the existence of such parametrizations since they do not exist in general for $d-$rectifiable sets.

\section{Preliminaries}\label{sec:prelim}

In this section we provide a brief background on rectifiable sets. For a more thorough treatment of the material, we refer the reader to \cite{simon-gmt}. First, we recall the definition of $d-$rectifiability.

\begin{definition}[$d-$Rectifiable]
    A set $\O\subset\R^D$ is $d-$rectifiable if it can be decomposed as $\O = \O_0 \cup(\cup_{j\in\N}F_j(A_j))$ where $\cH^d(\O_0) = 0,$ $A_j\subset \R^d$ and each map $F_j:\R^d\to\R^D$ is Lipschitz.
\end{definition}

Rectifiable subsets of $\R^D$ can be viewed as a natural measure-theoretic generalization of embedded submanifolds of $\R^D.$ Indeed, for each point $x\in\O$ there is some neighbourhood that is covered by countably many embedded $d-$dimensional $C^1-$submanifolds of $\R^D.$ An alternative characterization of $d-$rectifiable sets arises by considering {\em approximate tangent spaces}.

\begin{definition}[Approximate tangent space]
    Let $\O\subset\R^D$ be $\cH^d-$measurable and $\theta$ a positive $\cH^d-$measurable function on $\O$ with $\int_{\O\cap K}\theta d\cH^d<\infty$ for all compact sets $K\subset\R^D.$ For $x\in\O,$ the $d-$dimensional subspace $T_x\O\subset\R^D$ is an approximate tangent space to $\O$ with respect to $\theta$ if the equality
    \begin{equation}\label{eq:approx-tangent}
        \lim_{h\to 0 }\int_{(\O-x)/h}f(y)\theta(x+hy)d\cH^d(y) = \theta(x)\int_{T_x\O}f(y)d\cH^d(y)
    \end{equation}
    holds for all continuous compactly supported functions $f.$ Here, the set $(\O-x)/h = \{y\in \R^D : x+hy\in\O\}$ is the {\em blow up} of $\O$ around $x.$ 
\end{definition}

Note that any tangent space is also an approximate tangent space. In particular, if a set $\O$ is locally a differentiable manifold at $x,$ then $\O$ has an approximate tangent space at $x.$ The following theorem shows that a $d-$rectifiable set $\O\subset\R^D$ has approximate tangent spaces $\cH^d-$almost everywhere.

\begin{theorem}[{\cite[Theorem 3.1.9]{simon-gmt}}]\label{thm:rect-iff-tangent}
    Suppose $\O\subset\R^D$ is $\cH^d-$measurable and $\theta$ a positive $\cH^d-$measurable function on $\O$ with $\int_{\O\cap K}\theta d\cH^d<\infty$ for all compact sets $K\subset\R^D.$ Then, $\O$ is $d-$rectifiable if and only if $\O$ has an approximate tangent space $T_x\O$ with respect to $\theta$ for $\cH^d-$almost every $x\in\O.$
\end{theorem}

Theorem \ref{thm:rect-iff-tangent} will be the main tool necessary in the proof of Theorem \ref{thm:main-rectifiable}. We will see that the rate at which the mean square error $\MSE{\p_n(x)}$ converges to 0 is controlled by the rate of convergence in \eqref{eq:approx-tangent}. Since any tangent space to $\O$ is also an approximate tangent space, if $\O$ is locally a $d-$dimensional differentiable manifold almost everywhere then $\O$ is $d-$rectifiable as a direct consequence of Theorem \ref{thm:rect-iff-tangent}. In particular, objects such as $d-$dimensional algebraic varieties and semi-algebraic sets are also $d-$rectifiable sets. %

\begin{lemma}\label{lem:whitney}
    Suppose $\O$ is a $d-$dimensional algebraic variety or semi-algebraic set. Then, $\O$ is $d-$rectifiable.
\end{lemma}

\begin{proof}
    If $\O$ is a $d-$dimensional algebraic variety or semi-algebraic set, there exists a {\em Whitney stratification} of $\O$ - a decomposition of $\O$ into finitely many non-interecting smooth manifolds with dimension at most $d$ that satisfy certain compatibility conditions, see \cite{whitney}. In particular, for $\cH^d-$almost every $x\in\O,$ there exists a $d-$dimensional tangent space $T_x\O$ to $\O$ at $x.$ Since $T_x\O$ is also an {\em approximate} tangent space, $\O$ is $d-$rectifiable.
\end{proof}

\section{Proof of Theorem \ref{thm:main-rectifiable}}\label{sec:proof-rectifiable}

In this section, we prove Theorem \ref{thm:main-rectifiable}. First, we fix some assumptions on the set $\O$ and the kernel $K.$

{\bf Assumptions.} We fix the following setting:
\begin{enumerate}[label={A.\arabic*}]
    \item The kernel $K$ is continuous and vanishes outside of the interval $[0, 1],$ and the equality $\int_{\R^d}K(\|u\|)du = 1$ holds.
    \item \label{ass:tangents} The domain $\O\subset\R^D$ is $d-$rectifiable and the approximate tangent spaces to $\O$ with respect to $p$ satisfy
    \begin{equation*}
    \int_{(\O-x)/h}f(y)p(x+hy)d\cH^d(y) = p(x)\int_{T_x\O}f(y)d\cH^d(y) + O(h^m)
    \end{equation*}
    for some fixed $m>0$ and each compactly supported continuous function $f.$
\end{enumerate}

The assumption on $K$ is standard and is satisfied by most reasonable examples. The assumption on the approximate tangent spaces of $\O$ is necessary to quantify the rate at which the mean squared error $\MSE{\p_n(x)}$ converges to zero as the number of samples $n$ grows. Note that without this assumption, we may still conclude that the equality $\lim_{n\to\infty}\MSE{\p_n(x)} = 0$ holds, but cannot quantify the speed at which this convergence occurs.

Recall that the mean square error of the predictor $\p_n(x)$ can be decomposed into its squared bias and variance:

\begin{align*}
    \MSE{\p_n(x)} &= \Exp{\left(\Exp{\p_n(x)} - p(x)\right)^2} \\
    &= \left(\Exp{\p_n(x)} - p(x)\right)^2 + \Exp{\left(\Exp{\p_n(x)}-p(x)\right)^2}\\
    &=\Bias{\p_n(x)}^2 + \Var{\p_n(x)}
\end{align*}

We prove that both the bias and variance of $\p_n$ converge to zero as the number of samples $n$ grows. This argument is similar in spirit to those of \cite{parzen}, \cite{pelletier2005riemannian}, and \cite{ozakin2009submanifold}. We begin by considering the bias $\Bias{\p_n(x)}.$

\begin{lemma}\label{lem:unbiased}
    For $\cH^d-$almost every $x\in\O$ the equality
    $$\frac{1}{h^d}\int_{\O}K\left(\frac{\|x-y\|}{h}\right)dP(y) = p(x) + O(h^m).$$
    holds. In particular, if $h_n$ is chosen such that $\lim_{n\to\infty}h_n= 0,$ then $\Bias{\p_n(x)} = O(h_n^{m})$ with probability 1.
\end{lemma}
\begin{proof}
    Since $p$ is the density of the probability measure $P$ with respect to the Hausdorff measure $\cH^d,$ the equalities
    \begin{align*}
        \frac{1}{h^d}\int_{\O}K\left(\frac{\|x-y\|}{h}\right)dP(y) &= \frac{1}{h^d}\int_\O K\left(\frac{\|x-y\|}{h}\right)p(y)dP(y) \\
        &= \int_{(\O-x)/h}K(\|u\|)p(x+uh)d\cH^d(u)
    \end{align*}
    hold, with the second equality following from the change of variables $u = \frac{1}{h}(x-y).$ As $\O$ is $d-$rectifiable, there exists an approximate tangent space $T_x\O$ with respect to the density $p$ for almost every $x$ in $\O$ and the equality
    \begin{align*}
        \int_{(\O-x)/h}K(\|u\|)p(x+uh)d\cH^d(u) = p(x)\int_{T_x\O}K(\|u\|)d\cH^d(u) + O(h^m)
    \end{align*}
    holds by assumption \ref{ass:tangents} since $K$ is a compactly supported continuous function. Now, since $\O$ is assumed to be a subset of $\R^D,$ the approximate tangent space $T_x\O$ is a $d-$dimensional subspace of $\R^D.$ So, there exists some isometry mapping $T_x\O$ to $\R^d.$ It follows that the equalities
    \begin{align*}
        \int_{T_x\O}K(\|u\|)d\cH^d(u) &= \int_{\R^d}K(\|u\|)d\cH^d(u)=1
    \end{align*}
    hold. Thus, we have that $\frac{1}{h^d}\int_{\O}K\left(\frac{\|x-y\|}{h}\right)dP(y) - p(x) = O(h^m)$ holds for almost every $x\in\O.$ 
    
    We now show that with probability 1, the equality $\Bias{\p_n(x)} = O(h_n^m)$ holds. Recall that by the definition of bias,
    \begin{align*}
        \Bias{\p_n(x)} &= \frac{1}{h_n^d}\Exp{K\left(\frac{\|x-y\|}{h_n}\right)} - p(x)\\
        &= \frac{1}{h_n^d}\int_\O K\left(\frac{\|x-y\|}{h_n}\right)dP(y) - p(x)
    \end{align*}
    By our previous argument, we have that $\Bias{\p_n(x)} = O(h_n^m).$ This concludes the proof.
\end{proof}

We now utilize similar arguments to show that the variance $\Var{\p_n(x)}$ converges to zero as the number of samples $n$ increases.

\begin{lemma}\label{lem:variance}
    Suppose that $h_n$ is chosen such that $\lim_{n\to\infty}h_n=0$ and $\lim_{n\to\infty}nh_n^d=\infty.$ Then, with probability 1, the equality holds for all $h<h_0$:
    $$\Var{\p_n(x)} = O\left(\frac{1}{nh_n^d}\right).$$
\end{lemma}
\begin{proof}
    Recall that by the definition of variance,
    \begin{equation}\label{eq:var}
        \Var{\p_n(x)} = \frac{1}{n}\Exp{\frac{1}{h_n^{2d}}K^2\left(\frac{\|x-y\|}{h_n}\right)} - \frac{1}{n}\left(\Exp{\frac{1}{h_n^d}K\left(\frac{\|x-y\|}{h_n}\right)}\right)^2.
    \end{equation}
    We will consider each term in the sum independently. First, we consider the first term in \eqref{eq:var}
    \begin{align*}
        \frac{1}{n}\Exp{\frac{1}{h_n^{2d}}K^2\left(\frac{\|x-y\|}{h_n}\right)} &= \frac{1}{nh_n^d}\int_{\O}\frac{1}{h_n^d}K^2\left(\frac{\|x-y\|}{h_n}\right)p(y)dP(y).
    \end{align*}
    Using a similar argument to that in the proof of Lemma \ref{lem:unbiased}, we compute:
    \begin{align*}
        \frac{1}{nh_n^d}\int_{\O}\frac{1}{h_n^d}K^2\left(\frac{\|x-y\|}{h_n}\right)p(y)dP(y) = \frac{1}{nh_n^d}\left(\int_{\R^d}K^2(\|u\|)du +O(h_n^m)\right) \leq \frac{C p(x)}{nh_n^d} + O\left(\frac{h_n^m}{nh_n^{d}}\right),
    \end{align*}
    where $C \geq\int_{\R^d}K^2(\|u\|)du.$ Note that such a constant $C$ exists since $K$ is continuous and compactly supported. 
    By assumption, $\lim_{n\to\infty}h_n= 0$ and $\lim_{n\to\infty}nh_n^d=\infty$ and so it follows that the first term in \eqref{eq:var} converges to zero at the rate $O(1/(nh_n^d)).$ 
    
    Now, observe that by the arguments in the proof of Lemma \ref{lem:unbiased}, the equality holds:
    \begin{align*}
        \left(\Exp{\frac{1}{h_n^d}K\left(\frac{\|x-y\|}{h_n}\right)}\right)^2 = p^2(x) + O(h_n^m).
    \end{align*}
    Thus, dividing by the number of samples $n,$ the second term in \eqref{eq:var} is $O(1/n)$ as $n$ increases. Since we may assume that $h_n<1,$ note that $O(1/n)\subset O(1/(nh_n^d)).$ It follows that the equality $$\Var{\p_n(x)} = O\left(\frac{1}{nh_n^d}\right)$$ holds, as desired.
\end{proof}

The proof of Theorem \ref{thm:main-rectifiable} is now a simple combination of Lemmas \ref{lem:unbiased} and \ref{lem:variance}.

{\em Proof of Theorem \ref{thm:main-rectifiable}.} Since the mean square error of $\p_n(x)$ can be decomposed into squared bias and variance, we compute:
\begin{align*}
    \MSE{\p_n(x)} &= \Bias{\p_n(x)}^2 +\Var{\p_n(x)} \\
    &= O\left(h_n^{2m} + \frac{1}{nh_n^d}\right)
\end{align*}
Choosing the bandwidth parameter $h_n$ in proportion to $n^{-1/(d+2m)},$ the equality 
\begin{align*}
    \lim_{n\to\infty}\MSE{\p_n(x)} = O\left(\frac{1}{n^{2m/(d+2m)}}\right)
\end{align*}
holds, completing the proof. \qed

\section{Proof of Theorem \ref{thm:settings}}\label{sec:settings}

In this section, we prove that the convergence rate \eqref{eq:convergence-rate} holds when $\O$ is locally a smooth manifold almost everywhere. To this end, we introduce the following lemma, which quantifies the approximation error when integrating over the tangent space $T_x\O$ as opposed to the blow up $(\O-x)/h.$ The proof uses standard techniques but is included for completeness. %

\begin{lemma}\label{lem:convergence}
    Suppose there exists some $h_0>0$ such that $B(x, h_0)\cap\O$ is a smooth $d-$dimensional submanifold of $\R^D.$ Let $f$ be even, twice differentiable and compactly supported and $\theta$ a twice differentiable function on $\O.$ Then, the equality holds:
    $$\int_{(\O-x)/h}f(y)\theta(x+hy)d\cH^d(y) = \int_{T_x\O}f(y)\theta(x)d\cH^d(y) + O(h^2)$$
\end{lemma}

\begin{proof}
    As the Hausdorff measure is preserved under affine isometries, we may assume without loss of generality that $x = 0$ and $T_x\O=\R^d\subset\R^D,$ i.e. $T_x\O = \{(z, 0)\in\R^D:z\in\R^d\}.$ So, we prove the following equality:
\begin{equation}
    \int_{\O/h}f(y)\theta(hy)d\cH^d(y) = \int_{\R^d}f(z, 0)\theta(0)dz + O(h^2)
\end{equation}

Note that since $f$ is assumed to be compactly supported, we may intersect the domain of the integral with some compact set. Thus, we may assume that all functions and derivatives are bounded on our domain. First, we expand the integrand using the Taylor expansion of the function $\theta.$ Since $\theta$ is assumed to be twice differentiable, the equality holds:
\begin{equation}\label{eq:theta-expand-integral}
    \int_{\O/h}f(y)\theta(hy)d\cH^d(y) = \theta(0)\int_{\O/h}f(y)d\cH^d(y) + h\int_{\O/h} f(y)\langle D\theta(0), y\rangle d\cH^d(y) + O(h^2)
\end{equation}

We will independently consider each integral in the sum. Since $\O$ is a smooth manifold about $0,$ for small enough $h$ we may assume that the domain $\O/h$ can be locally parametrized as $\{(y/h), \phi(y)/h: y\in U\}$ where $\phi$ is a smooth map with $\phi(0)=0$ and $D\phi(0)=0$ and $U$ is some neighbourhood of the origin in $\R^d.$ Define the map $\psi_h:\R^d\to\R^D$ as $\psi_h(z) = (z, \phi(hz)/h).$ We now perform a change of variables using this parametrization of $\O/h:$
\begin{align*}
    \int_{\O/h}f(y)d\cH^d(y) &= \int_{\R^d}f(z, \phi(hz)/h)J_h(z) d\cH^d(y)
\end{align*}
where $J_h(y) = \sqrt{\det(D\psi_h(z)^\T D\psi_h(z))}$ is the Jacobian determinant of the map $\psi_h.$ By definition, the differential $D\psi_h(z)$ is 
\begin{align*}
    D\psi_h(z) &= \begin{bmatrix}
        I_d \\ D\phi(hz)
    \end{bmatrix}
\end{align*}
where $I_d$ is the $d\times d$ identity matrix. We approximate $D\phi(hz)$ with a second order Taylor approximation around $0.$ First, note that the Taylor expansion of $\phi(hz)$ is given by 
\begin{equation}\label{eq:phi-taylor}
    \phi(hz) = \phi(0) + h D\phi(0)(z) + \frac{h^2}{2} D^2\phi(0)(z, z) + O(h^3)
\end{equation}
 and so taking the derivative with respect to $z$,
\begin{align*}
    D\phi(hz) &= D\phi(0) + h^2 D^2\phi(0)(z) + O(h^3) = h^2 D^2\phi(0)(z) + O(h^3)
\end{align*}
since $D\phi(0)=0$ by assumption. So, $D\psi_h(z)^\T D\psi_h(z) = I_d + h^4 (D\phi(0)(z))^\T(D\phi(0)(z)) + O(h^5).$ Using the Taylor expansion of the determinant near the identity, we compute
\begin{align*}
    \det(I_d+ h^4 (D\phi(0)(z))^\T(D\phi(0)(z)) + O(h^5)) &= 1 + h^4\tr((D\phi(0)(z))^\T(D\phi(0)(z))) + O(h^4)\\
    &= 1 + O(h^4).
\end{align*}
It follows that the equality
\begin{align*}
    \sqrt{\det(D\psi_h(z)^\T D\psi_h(z))} &= 1 + O(h^4)
\end{align*}
holds. In particular, we have shown that the equality holds:
\begin{align*}
    \int_{\O/h}f(y)d\cH^d(y) &= \int_{\R^d}f(z, \phi(hz)/h)dz + O(h^2)
\end{align*}
We now compute the Taylor expansion of $f(z, \phi(hz)/h)$ in the normal direction to $\R^d\subset\R^D.$ Given the Taylor expansion of $\phi(hz)$ in \eqref{eq:phi-taylor}, it follows that the equality
\begin{equation}\label{eq:f-expansion}
    f(z, \phi(hz)/h) = f(z, 0) + \frac{h}{2}\langle D_{\perp}f(z, 0), D^2\phi(0)(z,z)\rangle + O(h^2)
\end{equation}
holds, where $D_\perp f(z, 0)$ is the derivative of $f$ at $(z,0)$ in the normal direction. Note that since $f$ is assumed to be even, $D_\perp f(z, 0)$ is an odd function of $z.$ We compute
\begin{align*}
    \int_{\R^d}f(z, \phi(hz)/h)dz &= \int_{\R^d} f(z, 0)dz  + \frac{h}{2}\int_{\R^d}\langle D_{\perp}f(z, 0), D^2\phi(0)(z,z)\rangle dz + O(h^2) \\
    &=\int_{\R^d} f(z, 0)dz + O(h^2),
\end{align*}
where the final equality holds since the inner product $\langle D_{\perp}f(z, 0), D^2\phi(0)(z,z)\rangle$ is an odd function of $z.$ Thus, we have shown that the equality
\begin{align*}
    \theta(0)\int_{\O/h}f(y)d\cH^d(y) = \theta(0)\int_{\R^d}f(z, 0)dz + O(h^2)
\end{align*}
holds. We now need only show that the second term in the sum \eqref{eq:theta-expand-integral} is of the order $O(h^2).$ Again, we use the change of variables $y = (z, \phi(hz)/h):$
\begin{align*}
    h\int_{\O/h}f(y)\langle D\theta(0), y\rangle d\cH^d(y) &= h\int_{\R^d}f(z, \phi(hz)/h)\langle D\theta(0), (z, \phi(hz)/h)\rangle dz + O(h^2)
\end{align*}

Using the Taylor expansion of $\phi(hz)/h$ in \eqref{eq:phi-taylor} note that the equalities hold:
\begin{align*}
    \langle D\theta(0), (z, \phi(hz)/h)\rangle &= \langle D\theta(0), (z, 0) +(0, \frac{h}{2}D^2\phi(0)(z, z)+ O(h^2))\rangle \\
    &=\langle D\theta(0), (z, 0)\rangle + \frac{h}{2}\langle D\theta(0), (0, D^2\phi(0)(z, z))\rangle + O(h^2)
\end{align*}
Taking the product with the expansion of $f(z, \phi(hz)/h)$ in \eqref{eq:f-expansion}, we compute
\begin{align*}
    h\int_{\R^d}f(z, \phi(hz)/h)\langle D\theta(0), (z, \phi(hz)/h)\rangle dz
    &=h\int_{\R^d}f(z, 0)\langle D\theta(0), (z, 0)\rangle dz + O(h^2)%
\end{align*}
Since $f$ is an even function by assumption, the product $f(z, 0)\langle D\theta(0), (z, 0)\rangle$ is an odd function of $z$. It follows that the equality 
\begin{align*}
    \int_{\R^d}f(z, 0)\langle D\theta(0), (z, 0)\rangle dz = 0
\end{align*}
holds and thus only the $O(h^2)$ term remains. This completes the proof.
\end{proof}

{\em Proof of Theorem \ref{thm:settings}.} By Lemma \ref{lem:convergence}, the convergence rate in Theorem \ref{thm:main-rectifiable} holds with $m=2.$ It follows that the equality
\begin{align*}
    \MSE{\p_n(x)} = O\left(\frac{1}{n^{4/(d+4)}}\right)
\end{align*}
holds for almost every $x\in\O$ when $h_n$ is chosen in proportion to $n^{-1/(d+4)}$ as desired. \qed

We have shown that the kernel density estimator $\p_n(x)$ converges to to the true density $p(x)$ at the rate $O(n^{-4/(d+4)})$ on sets that are locally smooth $d-$dimensional manifolds, mirroring the classical results of \cite{parzen,cacoullos1966estimation}. By Lemma \ref{lem:whitney}, this includes all $d-$dimensional algebraic varieties and semi-algebraic sets. This is a much more rich class of spaces than the set of all $d-$dimensional manifolds. For example, the spaces of sparse vectors and low-rank matrices are algebraic varieties and the space of positive semi-definite matrices is a semi-algebraic set. Our result ensures that kernel density estimation can be used to estimate unknown density functions on these spaces and converges to the true denstiy at the same rate as classical kernel density estimation. The next section illustrates this convergence rate by estimating a known density function on sparse vectors.

\section{Example: Density Estimation on Sparse Data}\label{sec:numerics}

One common structure used in dimensionality reduction is assuming that data is $d-$sparse: although the data lives in the high-dimensional space $\R^D,$ each sample only has $d$ non-zero coordinates where $d$ is typically much smaller than $D.$ We now illustrate the convergence of the estimator $\p_n$ by estimating a probability density on $d-$sparse vectors.

We fix the following distribution on the set of $d-$sparse vectors in $\R^D:$ data is sampled from a mixture of Gaussians by uniformly at randomly selecting a $d-$dimensional coordiante hyperplane and randomly sampling a vector according to the standard $d-$dimensional normal distribution in the hyperplane. For a $d-$sparse vector $x\in\R^D,$ the true density of this distribution at the point $x$ is given by
\begin{equation}
    p(x) = \frac{1}{{D\choose d}}\cdot\frac{\exp\left(-\|x\|^2/2\right)}{\sqrt{(2\pi)^d}}.
\end{equation}
since there are ${D\choose d}$ $d-$dimensional coordinate planes and each hyperplane is chosen uniformly at random.

To estimate the density $p$ we use the Kernel density estimator $\p_n(x)$ with $K$ chosen to be the truncated Gaussian kernel
\begin{equation*}
    K\left(\frac{\|x-y\|}{h}\right) = \begin{cases}\frac{1}{h\sqrt{2\pi}(F(1)-F(-1))}\cdot \exp\left(-\frac{\|x-y\|^2}{2h^2}\right) & \|x-y\|\leq h \\
        0 & \|x-y\|>h
    \end{cases}
\end{equation*}
where $F$ is the cumulative distribution function of the standard normal distribution. We use the truncated Gaussian kernel rather than the standard Gaussian kernel since it is compactly supported on the interval $[0, 1].$

\begin{figure}[h]
    \centering
    \includegraphics[width=0.65\textwidth]{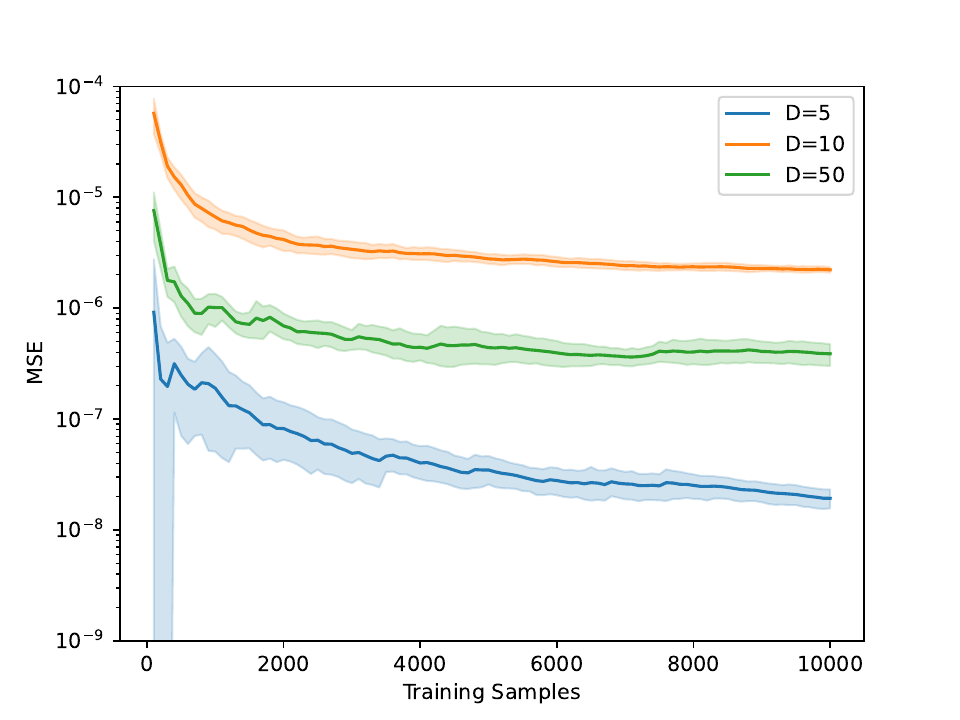}
    \caption{The empirical mean square error of the predictor $\p_n$ in various ambient dimensions and using varying sizes of training sets. The shaded regions indicate a 95\% confidence interval. Note that the curve corresponding to each ambient dimension $D$ has the same shape, illustrating that the convergence rate is independent of the ambient dimension and depends only on the intrinsic dimension $d.$}
    \label{fig:sparse-data}
\end{figure}

Note that the set of $d-$sparse vectors is a $d-$dimensional variety in $\R^D$ since it is the vanishing locus of the set of all degree $d+1$ monomials. Thus, from the result of Theorem \ref{thm:settings} we expect the MSE of the estimator $\p_n(x)$ to converge to 0 at the rate $O(n^{-4/(d+4)})$ regardless of the ambient dimension $D$ of the data. We fix the level of sparsity of the data at $d=3$ and test the estimator $\p_n$ in ambient dimensions $D=5, 10, 50.$ In each setting, we construct an estimator using training sets of various sizes: we begin with using only 100 training points and increase the number of data points in increments of 100 until the training set consists of 10,000 points. For a training set of size $n,$ the bandwidth parameter $h_n$ was chosen to be $n^{-1/7}$ in accordance with the choice of parameter in Theorem \ref{thm:settings}. We calculate the empirical mean square error of these predictors using a test set of 500 points. For each regime, we repeat the experiment 10 times. The results of this experiment are contained in Figure \ref{fig:sparse-data}.

\printbibliography

\end{document}